\documentclass[twoside, a4paper, 10pt]{amsart}
\title[ ]{Spherical recurrence and locally isometric embeddings of trees into positive density subsets of $\mathbb{Z}^d$}
\usepackage{amsaddr}
\author{Kamil Bulinski}
\address{School of Mathematics and Statistics F07, University of Sydney, NSW 2006, Australia}
\email{K.Bulinski@maths.usyd.edu.au}
\usepackage{amsfonts}
\usepackage{amsthm}
\usepackage{verbatim}
\usepackage{amsmath, amssymb}
\usepackage{tikz}
\usetikzlibrary{matrix, arrows}
\usepackage{listings}
\usepackage{color}
\usepackage{listings}
\usepackage[all]{xy}
\usepackage[pdftex,colorlinks,linkcolor=blue,citecolor=blue]{hyperref}
\usepackage{graphicx}
\usepackage{float}
\usepackage[margin=3cm]{geometry}
\usepackage{bigints}
\usepackage{dsfont}
\begin{document}
\maketitle
\raggedbottom

%% Mathcal large
\newcommand{\cA}{\mathcal{A}}
\newcommand{\cB}{\mathcal{B}}
\newcommand{\cC}{\mathcal{C}}
\newcommand{\cD}{\mathcal{D}}
\newcommand{\cE}{\mathcal{E}}
\newcommand{\cF}{\mathcal{F}}
\newcommand{\cG}{\mathcal{G}}
\newcommand{\cH}{\mathcal{H}}
\newcommand{\cI}{\mathcal{I}}
\newcommand{\cJ}{\mathcal{J}}
\newcommand{\cK}{\mathcal{K}}
\newcommand{\cL}{\mathcal{L}}
\newcommand{\cM}{\mathcal{M}}
\newcommand{\cN}{\mathcal{N}}
\newcommand{\cO}{\mathcal{O}}
\newcommand{\cP}{\mathcal{P}}
\newcommand{\cQ}{\mathcal{Q}}
\newcommand{\cR}{\mathcal{R}}
\newcommand{\cS}{\mathcal{S}}
\newcommand{\cT}{\mathcal{T}}
\newcommand{\cU}{\mathcal{U}}
\newcommand{\cV}{\mathcal{V}}
\newcommand{\cW}{\mathcal{W}}
\newcommand{\cX}{\mathcal{X}}
\newcommand{\cY}{\mathcal{Y}}
\newcommand{\cZ}{\mathcal{Z}}
%% Mathbb large
\newcommand{\bA}{\mathbb{A}}
\newcommand{\bB}{\mathbb{B}}
\newcommand{\bC}{\mathbb{C}}
\newcommand{\bD}{\mathbb{D}}
\newcommand{\bE}{\mathbb{E}}
\newcommand{\bF}{\mathbb{F}}
\newcommand{\bG}{\mathbb{G}}
\newcommand{\bH}{\mathbb{H}}
\newcommand{\bI}{\mathbb{I}}
\newcommand{\bJ}{\mathbb{J}}
\newcommand{\bK}{\mathbb{K}}
\newcommand{\bL}{\mathbb{L}}
\newcommand{\bM}{\mathbb{M}}
\newcommand{\bN}{\mathbb{N}}
\newcommand{\bO}{\mathbb{O}}
\newcommand{\bP}{\mathbb{P}}
\newcommand{\bQ}{\mathbb{Q}}
\newcommand{\bR}{\mathbb{R}}
\newcommand{\bS}{\mathbb{S}}
\newcommand{\bT}{\mathbb{T}}
\newcommand{\bU}{\mathbb{U}}
\newcommand{\bV}{\mathbb{V}}
\newcommand{\bW}{\mathbb{W}}
\newcommand{\bX}{\mathbb{X}}
\newcommand{\bY}{\mathbb{Y}}
\newcommand{\bZ}{\mathbb{Z}}

\newcounter{dummy} \numberwithin{dummy}{section}

\theoremstyle{definition}
\newtheorem{mydef}[dummy]{Definition}
\newtheorem{prop}[dummy]{Proposition}
\newtheorem{corol}[dummy]{Corollary}
\newtheorem{thm}[dummy]{Theorem}
\newtheorem{lemma}[dummy]{Lemma}
\newtheorem{eg}[dummy]{Example}
\newtheorem{notation}[dummy]{Notation}
\newtheorem{remark}[dummy]{Remark}
\newtheorem{claim}[dummy]{Claim}
\newtheorem{Exercise}[dummy]{Exercise}
\newtheorem{question}[dummy]{Question}

\begin{abstract} Magyar has shown that if $B \subset \bZ^d$ has positive upper density $(d \geq 5)$, then the set of squared distances $\{ \|b_1-b_2 \|^2 \text{ }: \text{ } b_1,b_2 \in B \}$ contains an infinitely long arithmetic progression, whose period depends only on the upper density of $B$. We extend this result by showing that $B$ contains locally isometrically embedded copies of every tree with edge lengths in some given arithmetic progression (whose period depends only on the upper density of $B$ and the number of vertices of the sought tree). In particular, $B$ contains all chains of elements with gaps in some given arithmetic progression (which depends on the length of the sought chain). This is a discrete analogue of a result obtained recently by Bennett, Iosevich and Taylor on chains with prescribed gaps in sets of large Haussdorf dimension. Our techniques are Ergodic theoretic and may be of independent interest to Ergodic theorists. In particular, we obtain Ergodic theoretic analogues of recent \textit{optimal spherical distribution} results of Lyall and Magyar which, via Furstenberg's correspondence principle, recover their combinatorial results.  \end{abstract}

\section{Introduction}

\subsection{Combinatorial results}

We recall the following result of Magyar on the existence of infinite arithmetic progressions in the distance sets of positive density subsets of $\bZ^d$, for $d \geq 5$.

\begin{thm}[Magyar \cite{Magyar}]\label{Magyar original} For all $\epsilon>0$ and integers $d \geq 5$ there exists a positive integer $q=q(\epsilon,d)$ such that the following holds: If $B \subset \bZ^d$ has upper Banach density $$d^*(B):= \lim_{L \to \infty} \max_{x \in \bZ^d} \frac{|B \cap (x+[0,L)^d)|}{L^d}> \epsilon,$$ then there exists a positive integer $N_0=N_0(B)$ such that\footnote{In this paper, $\|v\|=\|v\|_2$ will always mean the standard Euclidean $L^2$-norm.} $$ q\bZ_{>N_0} \subset \left\{ \| b_1 - b_2\|^2 \text{ } | \text{ } b_1,b_2 \in B \right\}.$$\end{thm}

Our main goal is to prove the existence of other infinite families of Euclidean configurations in positive density subsets of $\bZ^d$. Before we state the main result, let us provide an instructive example of a special case.

\begin{prop}[Chains with prescribed gaps]\label{prop: chains}  For all $\epsilon>0$ and integers $m>0$ and $d \geq 5$ there exists a positive integer $q=q(\epsilon,m,d)>0$ such that the following holds: For all $B \subset \bZ^d$ with $d^*(B)>\epsilon$ there exists a positive integer $N_0=N_0(B,m)$ such that for all integers $t_1, \ldots, t_m\geq N_0$ there exist \textbf{distinct} $b_1, \ldots, b_{m+1} \in B$ such that $$\|b_{i}-b_{i+1}\|^2=qt_i \text{ for all $i=1,\ldots, m$.} $$  \end{prop}

In other words, we may \textit{locally isometrically} embed, into a subset $B \subset \bZ^d$ of positive density, all path graphs with edge lengths in some infinite arithmetic progression, whose common difference depends only on the upper Banach density of $B$ and the number of vertices of the given path graph (as well as the dimension $d$). Continuous analogues of such results were recently obtained by Bennett, Iosevich and Taylor in \cite{IosevichChains}, where they prove that a set of large enough Haussdorf dimension in Euclidean space contains long chains whose gaps can assume any value in some fixed open interval of $\bR$ (which depends on the length of the sought chain).

More generally, we show that Proposition~\ref{prop: chains} is true if one replaces path graphs with arbitrary trees. To make this statement precise, we introduce the following terminology.

\begin{mydef}[Edge labelled trees and local isometries]\label{def: intro trees} An edge-labelled tree is a tuple $\tau=(V,E,\phi,L)$ where $(V,E)$ is a finite tree (connected acyclic graph) with vertex set $V$ and edge set $E$ and $\phi:E \to L$ is a function to a set $L$. If $L \subset \bZ_{>0}$ then a \textit{locally isometric embedding} of $\tau$ into $\bZ^d$ is an injective map $\iota: V \to \bZ^d$ such that for each edge $e=\{v_1,v_2\} \in E$ we have that $$\|\iota(v_1)-\iota(v_2)\|^2=\phi(e).$$ In other words, one should think of $\phi(e)$ as being the square of the length of the edge $e$ and $\iota$ as being an embedding that preserves distances between adjacent vertices (but not neccesarily the distance between non-adjacent vertices, hence $\iota$ is only a \textbf{local} isometry).

\end{mydef}

 Our main result may now be stated as follows.

\begin{thm}\label{thm: main comb thm intro} Let $\epsilon>0$, $d \geq 5$ and suppose that $B \subset \bZ^d$ with $d^{\ast}(B)>\epsilon$. Then for each positive integer $m$, there exists a positive integer $q=q(\epsilon,m,d)$ and a positive integer $N_0=N_0(B,m)$ such that whenever $\tau=(V,E,\phi,\bZ_{>N_0})$ is an edge-labelled tree, with $|V|=m$, then there exists a locally isometric embedding $\iota$ of $\tau$ into $\bZ^d$ such that $$q\cdot \iota (V) \subset B - b \quad \text{for some } b \in B.$$ In particular, all edge-labelled trees of the form $\tau=(V,E,\phi, q^2\bZ_{>N_0})$, with $|V|=m$, may be locally isometrically embedded into the set $B$.

\end{thm}

By considering the trees of diameter $2$ (i.e., the trees of the form $(V,E)$ where $V=\{v_0, \ldots, v_{m-1}\}$ and $E=\{ \{v_0,v_i\} \text{ } | \text{ } i=1, \ldots, m-1\}$) we recover a recent result of Lyall and Magyar on \textit{pinned distances} \cite{LyallMagyar}. We remark that their result for this particular family of trees is quantitatively superior to what we have stated as it turns out that, after restricting to this family of trees, the integer $q$ does not depend on $m$ (only on $\epsilon$ and $d$). See Section~\ref{sec: optimal recurrence} below for more details. This leads to the following unresolved question.

\begin{question} Does the positive integer $q$ in the conclusion of Theorem~\ref{thm: main comb thm intro} only depend on $\epsilon$ and $d$ (i.e., not on $m$) ? \end{question}

\subsection{Ergodic results}

The combinatorial results stated in the preceding section will be proven via an Ergodic theoretic approach. We will however aim to make the paper accessible to readers with no knowledge of Ergodic theory.

\begin{mydef}[Notational conventions and basic Ergodic theory] If $X$ is a set, then the statement \\ $T:\bZ^d \curvearrowright X$ indicates that there is a homomorphism $a \mapsto T^a$ from $\bZ^d$ to $\text{Aut}(X)$. In other words, $T$ is an action of $\bZ^d$ on $X$, where the action of $a \in \bZ^d$ on $x \in X$ is denoted by $T^a x$. Morever, given a positive integer $q \in \bZ^d$ we let $T^q:\bZ^d \curvearrowright X$ be the action given by $$(T^q)^a x = T^{qa} x \text{ for all } a\in \bZ^d \text{ and } x \in X.$$ If $X$ is a measurable space and $\mu$ is a probability measure on $X$ then we say that the action $T$ preserves $\mu$ if $$\mu(T^aB) =\mu(B) \text{ for all $a \in \bZ^d$ and (measurable}\footnote{From now on, we omit the adjective \textit{measurable} and simple take measurability to be an implicit assumption.}) \text{ }B \subset X.$$ We say that $T:\bZ^d \curvearrowright (X,\mu)$ is a measure preserving action if $(X,\mu)$ is a probability space and the action $T$ preserves $\mu$. Such a measure preserving action induces a unitary action on $L^2(X,\mu)$, which we shall also denote as $T:\bZ^d \curvearrowright (X,\mu)$, given by $$T^a f = f \circ T^a \text{ for } a \in \bZ^d \text{ and } f \in L^2(X,\mu).$$ We let $$L^2(X,\mu)^T = \{ f \in L^2(X,\mu) \text{ } | \text{ } T^a f = f \text{ for all } a \in \bZ^d \}$$ denote the space of $T$-invariant functions and we let $P_T:L^2(X,\mu) \to L^2(X,\mu)^T$ denote the orthogonal projection onto it. The action $T:\bZ^d \curvearrowright (X,\mu)$ is said to be \textit{ergodic} if $L^2(X,\mu)^T$ consists of only the constant functions, which is equivalent to the statement that there are no $B \subset X$, with $0<\mu(B)<1$, such that $T^aB=B$ for all $a \in \bZ^d$. 
\end{mydef}

We now state the standard Furstenberg correspondence princple which will allow us to reduce the combinatorial results stated above to Ergodic theoretic statements.

\begin{prop}[Furstenberg's correspondence principle, \cite{FurDiag}] Let $B \subset \bZ^d$. Then there exists a compact metric space $X$ together with an action $T:\bZ^d \curvearrowright X$ by homeomorphisms and a point $x_0 \in X$ and clopen $\widetilde{B} \subset X$ such that $$B = \{ b \in \bZ^d \text{ } | \text{ } T^bx_0 \in \widetilde{B} \}.$$ Moreover, there exists an ergodic $T$-invariant probability measure $\mu$ on $X$ such that $d^*(B)=\mu(\widetilde{B})$ and $$d^*\left( \bigcap_{a \in F} (a+ B)\right) \geq \mu \left(\bigcap_{a \in F} T^a \widetilde{B} \right) \text{ for all finite } F \subset \bZ^d.$$

\end{prop}

Hence Magyar's theorem above (Theorem~\ref{Magyar original}) may be reduced to the following \textit{spherical recurrence} result, which we will prove in Section~\ref{sec: spherical recurrence}.

\begin{prop} \label{prop: ergodic magyar} Fix $\epsilon>0$ and an integer $d \geq 5$. Then there exists a positive integer $q=q(\epsilon,d)$ such that the following holds: For all measure preserving systems $T:\bZ^d \curvearrowright (X,\mu)$ and $B \subset X$ with $\mu(B)>\epsilon$, there exists $N_0$ such that for all $N \geq N_0$ there exists $a \in \{ v \in \bZ^d \text{ }: \text{ } \|v \|^2 =N \}$ such that $$\mu(B \cap T^{qa}B) > 0.$$ 

\end{prop}

More generally, we will establish the following recurrence result which implies our main combinatorial result about locally isometrically embedded trees in positive density subsets (Theorem~\ref{thm: main comb thm intro}).

\begin{thm}\label{thm: intro tree recurrence}  Fix $\epsilon>0$ and integers $m \geq 2$ and $d \geq 5$. Then there exists a positive integer $q=q(\epsilon,m,d)$ such that the following holds: Suppose that $T:\bZ^d \curvearrowright (X,\mu)$ is a measure preserving action and $B \subset X$ with $\mu(B)>\epsilon$. Then there exists $N_1$ such that for all edge-labelled trees $\tau=(V,E,\phi,\bZ_{>N_1})$, with $|V|=m$, there exists a locally isometric embedding $\iota$ of $\tau$ into $\bZ^d$ such that $$\mu\left( \bigcap_{v \in V} T^{q \cdot \iota(v)}B \right) >0. $$

\end{thm}

We now finish with a \textit{pointwise recurence} result, whose combinatorial consequence is an \textit{optimal spherical distribution result} recently obtained by Lyall and Magyar (Theorem $1$ in \cite{LyallMagyar}). Let $S_N=\{ v \in \bZ^d \text{ }: \text{ } \|v \|^2=N \}$. We shall use the shorthand ``$P(b)$ holds for $\mu$-many $b \in B$'' to mean $\mu(\{ b \in B \text{ } | \text{ } P(b) \})>0$.

\begin{thm}\label{thm: intro pointwise recurrence}  Fix $\epsilon>0$ and an integer $d \geq 5$. Then there exists a positive integer $q=q(\epsilon,d)$ such that the following holds: Let $T: \bZ^d \curvearrowright (X,\mu)$ be an ergodic measure preserving action and $B \subset X$, with $\mu(B)> \epsilon$. Then there exists $N_0$ such that for all $N \geq N_0$ we have that $$\frac{1}{|S_N|} \sum_{a \in S_N} \mathds{1}_{B}(T^{qa}b) > \mu(B) - \epsilon \quad{ } \text{ for $\mu$-many} \text{ }  b \in B.$$ 

\textbf{Acknowledgement:} The author is grateful to Alex Iosevich, who posed a question on chains in sets of positive density, which led to Proposition~\ref{prop: chains}.

\end{thm}

\section{Main tool from discrete harmonic analysis}

We now state the main blackbox that we will use in our proofs, which is also used in Magyar's original proof \cite{Magyar} as well as the recent optimal improvements and pinned generalizations of Magyar and Lyall in \cite{LyallMagyar}. Fix an integer $d \geq 5$. For $N \in \bN$, we define the discrete sphere $$S_{N} = \{ x \in \bZ^d \text{ }| \text{ } \|x \|_2^2 = N \}.$$

For $\eta>0$ and $C>0$ let $$q_{\eta,C} = \text{lcm} \{ q \in \bZ \text{ } | \text{ } 1 \leq q \leq C\eta^{-2}  \}$$ 

Magyar-Stein-Wainger gave approximations for exponential sums on a discrete sphere \cite{MSW}. The following is a rather straightforward consequence of these approximations proved in Magyar's original paper \cite{Magyar} (this particular formulation is stated in a recent work of Lyall and Magyar \cite{LyallMagyar}).

\begin{thm}[Exponential sum estimates on discrete spheres]\label{thm: exponential sums} There exists a constant $C=C_d>0$ (depends only on $d \geq 5$) such that the following is true: Given  $\eta>0$, an integer $N \geq C\eta^{-4}$ and $$ \theta \in \bR^d \setminus \left( q_{\eta,C}^{-1} \bZ^d + [-(\eta N)^{-1/2}, (\eta N)^{-1/2} ] \right), $$ then $$\left| \frac{1}{|S_{N}|} \sum_{x \in S_{N}} \exp(2\pi i \langle x , \theta \rangle \right| \leq \eta.$$ 

\end{thm}

We stress that this consequence of \cite{MSW} has less than a one page proof in \cite{Magyar} and thus only takes up a small portion of that paper. As such, using this as a blackbox does not detract much from our alternative proof of Theorem~\ref{Magyar original} given in Section \ref{sec: spherical recurrence}.

\textbf{Important Notational Convention:} For brevity, let us now write $q_{\eta}=q_{\eta,C}$ where $C=C_d$ is as above (the dimension $d \geq 5$ will from now on be fixed).

\section{$(q,\delta)$-Equidistributed sets.}

In this section, we introduce the notion of a $(q,\delta)$-equidistributed subset of an Ergodic system, which may be of independent interest. This will enable us to employ a \textit{measure increment} argument from which we will obtain good control of the integer $q=q(\epsilon,m,d)$ appearing in Theorems~\ref{thm: main comb thm intro} and \ref{thm: intro tree recurrence} above. We briefly remark that a combinatorial analogue of such an increment argument (called the \textit{density increment} argument) is often used in Additive Combinatorics. In fact, it is used in Magyar's original proof of Theorem~\ref{Magyar original} as well as in the recent work of Lyall and Magyar \cite{LyallMagyar}. However, the details are slightly more technical in our Ergodic setting.

For the remainder of this section, let $F_N=[1,N]^d \cap \bZ^d$.

\begin{mydef} \label{def: equidistributed} Let $T: \bZ^d \curvearrowright (X,\mu)$ be an ergodic measure preserving action. Then we say that $B \subset X$ is \textit{$(q,\delta)$-equidistributed} if for almost all $x \in X$ we have $$ \lim_{N \to \infty} \frac{1}{|F_N|}\left | \{a \in F_N \text{ } | \text{ } T^{qa}x \in B \}\right| \leq (1+\delta)\mu(B). $$

\end{mydef}

\begin{mydef}[Conditional probability and ergodic components] \label{def: ergodic components of finite index} If $(X,\mu)$ is a probability space and $C \subset X$ is measurable with $\mu(C)>0$ then we define the conditional probability measure $\mu(\cdot | C)$ given by $\mu(B|C)=\frac{\mu(B \cap C)}{\mu(C)}$. We note that if $C$ is invariant under some measure preserving action, then $\mu(\cdot |C)$ is also preserved by this action. If $T: \bZ^d \curvearrowright (X,\mu)$ is ergodic and $Q$ is a positive integer, then the action $T^Q: \bZ^d \curvearrowright(X,\mu)$ may not be ergodic; but it is easy to see that there exists a $T^Q$-invariant subset $C \subset X$ such that the action of $T^Q$ on $C$ is ergodic (more precisely, $\mu( \cdot | C)$ is $T^Q$-ergodic) and the translates of $C$ disjointly cover $X$ (there are at most $Q^d$ distinct translates, hence $\mu(C) \geq Q^{-d}$). Note that the translates of $C$ also satisfy these properties of $C$. We call such a measure $\mu( \cdot | C)$ a $T^Q$-ergodic component of $\mu$. It follows that $\mu$ is the average of its distinct $T^Q$-ergodic components.

\end{mydef}

We may now introduce our measure increment technique, which will be used to reduce our recurrence theorems, such as Theorem~\ref{thm: intro tree recurrence}, to ones which assume sufficient equidistribution.

\begin{lemma}[Ergodic measure increment argument]\label{lemma: measure increment} Let $\delta,\epsilon>0$, let $q$ be a positive integer and let $T: \bZ^d \curvearrowright (X,\mu)$ be ergodic. If $B \subset X$ with $\mu(B)>\epsilon$ then there exists a positive integer $Q  \leq q^{\log(\epsilon^{-1})/\log(1+\delta)}$ and a $T^Q$-ergodic component, say $\nu$, of $\mu$ such that $\nu(B) \geq \mu(B)$ and $B$ is $(q,\delta)$-equidistributed with respect to $T^Q:\bZ^d \curvearrowright (X,\nu)$. 

\end{lemma}

To study the limits appearing in Definition~\ref{def: equidistributed} we make use of the well known Pointwise Ergodic Theorem.

\begin{prop}[Pointwise ergodic theorem] Let $T: \bZ^d \curvearrowright (X,\mu)$ be a measure preserving action. Then for all $f \in L^2(X,\mu)$ there exists $X_f \subset X$ with $\mu(X_f)=1$ such that $$ \lim_{N \to \infty} \frac{1}{|F_N|} \sum_{a \in F_N} f(T^a x) \to P_T f (x) $$ for all  $x \in X_f$.

\end{prop}

\begin{proof}[Proof of Lemma~\ref{lemma: measure increment}] If $B$ is $(q,\delta)$ equidistributed, then we are done. Otherwise, it follows from the Pointwise ergodic theorem (applied to the action $T^q$ and the indicator function of $B$) that there exists a $T^q$-ergodic component of $\mu$, say $\nu_1$, such that $\nu_1(B) \geq (1+\delta)\mu(B)$. Continuing in this fashion, we may produce a maximal sequence of Ergodic components $\nu_1,\nu_2 \ldots, \nu_J $ of $T^q,T^{q^2}, \ldots T^{q^J}$, respectively, such that $\nu_{j+1}(B) \geq (1+\delta) \nu_{j} (B)$. Clearly we must have $\epsilon(1+\delta)^J \leq 1$ and so this finishes the proof with $Q=q^J$. \end{proof}

We now turn to demonstrating the key spectral properties of a $(q,\delta)$-equidistributed set.

\begin{mydef}[Eigenspaces] \label{def: eigen} If $T: \bZ^d \curvearrowright (X,\mu)$ is a measure preserving action and $\chi \in \widehat{\bZ^d}$ is a character on $\bZ^d$, then we say that $f \in L^2(X,\mu)$ is a $\chi$-eigenfunction if $$T^a f= \chi(a)f \text{ for all } a \in \bZ^d.$$ We let $\text{Eig}_T(\chi)$ denote the space of $\chi$-eigenfunctions and for $R \subset  \widehat{\bZ^d}$ we let $$\text{Eig}_T(R)=\overline{\text{Span}\{ f \text{ } | \text{ } f \in \text{Eig}(\chi) \text{ for some } \chi \in R \}}^{L^2(X,\mu)}.$$
In particular, we will be intersted in the sets $R_q = \{ \chi \in \widehat{\bZ^d}  \text{ } |\text{ } \chi^q=1 \}$ and $R^*_q= R_q \setminus \{ 1 \}$, where $q \in \bZ$.  Note that the spaces $\text{Eig}_T(\chi)$ are orthogonal to eachother and hence $\text{Eig}_T(R)$ has an orthonormal basis consiting of $\chi$-eigenfunctions, for $\chi \in R$. Note also that Ergodicity implies that each $\text{Eig}_T(\chi)$ is at most one dimensional.\end{mydef}

\begin{prop} \label{prop: equidist implies small proj}  Let $T: \bZ^d \curvearrowright (X,\mu)$ be an ergodic measure preserving action and suppose that $B \subset X$ is $(q,\delta)$-equidistributed. Let $h \in L^2(X,\mu)$ be the orthogonal projection of $\mathds{1}_B$ onto $\text{Eig}_T(R^*_q)$. Then $$P_{T^q} \mathds{1}_B=\mu(B) + h$$ and $$\| h \|_2 \leq \sqrt{(2\delta + \delta^2)}\mu(B).  $$

\end{prop}

\begin{proof} Note that\footnote{This follows from the fact that all finite dimensional representations of a finite abelian group can be decomposed into one dimensional representations.} $\text{Eig}_T(R_q)=L^2(X,\mu)^{T^q}$. This, together with the ergodicity of $T$, shows that \\$h = P_{T^q}\mathds{1}_B - \mu(B)$. Now the pointwise ergodic theorem, applied to the action $T^q$, combined with the $(q,\delta)$-equidistribution of $B$ immediately gives that $$\| h \|^2_2 = \|P_{T^q} \mathds{1}_B \|^2 - \| \mu(B) \|^2_2 \leq (1+\delta)^2\mu(B)^2 - \mu(B)^2= (2\delta + \delta^2) \mu(B)^2.$$ \end{proof} 

\section{Spherical mean ergodic theorem}

Our next result says that the ergodic averages along the discrete spheres $S_N$ of a well enough equidistributed set $B$ must almost converge (that is, are eventually very close to) in $L^2$ to the constant function $\mu(B)$.

\begin{thm}[Spherical Mean Ergodic theorem] \label{spherical mean ergodic thm}  Let $T: \bZ^d \curvearrowright (X,\mu)$ be an ergodic measure preserving action and suppose that $B \subset X$ is $(q_{\eta},\delta)$-equidistributed. Then \begin{align}\label{spherical mean estimate} \limsup_{N \to \infty}\left \| \mu(B) - \frac{1}{|S_{N}|} \sum_{a \in S_{N}} T^a \mathds{1}_B\right \|_2 \leq \sqrt{3\delta} + \eta. \end{align}
\end{thm}

We first prove the following lemma using the spectral theorem. Using the notation introduced in Definition~\ref{def: eigen}, let $\textbf{Rat}=\bigcup_{q \in \bN} R_q$ denote the set of rational characters, let $L^2_{\textbf{Rat}}(X,\mu,T)=\textbf{Eig}_T(\textbf{Rat})$ denote the rational Kronecker factor and let $P_{\textbf{Rat}}:L^2(X,\mu) \to L^2_{\textbf{Rat}}(X,\mu,T)$ denote the orthogonal projection onto it.

\begin{lemma}\label{lemma: vanishing mean}  Let $T: \bZ^d \curvearrowright (X,\mu)$ be a measure preserving action and suppose that $f \in L^2_{\textbf{Rat}}(X,\mu,T)^{\perp}$. Then $$\lim_{N \to \infty} \frac{1}{|S_{N}|} \sum_{a \in S_{N}} T^af = 0.$$

\end{lemma}

\begin{proof} By the spectral theorem there exists a positive finite Borel measure $\sigma$ on $\bT^d$ such that $$\langle T^af, f \rangle = \int \exp(2\pi i u \cdot a) d\sigma(u) \text{ for all } a \in \bZ^d.$$ Since $f \in L^2_{\textbf{Rat}}(X,\mu,T)^{\perp}$ we have that $\sigma ( \bQ^d/\bZ^d)=0$ and hence $$ \|\frac{1}{|S_{N}|} \sum_{a \in S_{N}} T^af \|^2_2 =\int_{\Omega} \left| \frac{1}{|S_{N}|} \sum_{a \in S_{N}} \exp(2\pi i \langle u , a\rangle ) \right|^2 d\sigma(u)  $$ where $\Omega=\bT^d \setminus (\bQ^d /\bZ^d).$ But  $$\frac{1}{|S_{N}|} \sum_{a \in S_{N}} \exp(2\pi i \langle u , a\rangle) \to 0 \text{ as } N \to \infty$$ for all $u \in \Omega$ by Theorem~\ref{thm: exponential sums}. Now the dominated convergence theorem finally completes the proof. \end{proof}

\begin{proof}[Proof of Theorem~\ref{spherical mean ergodic thm}] Let $q=q_{\eta}$. By Lemma~\ref{lemma: vanishing mean}, the left hand side of (\ref{spherical mean estimate}) remains unchanged if we replace $\mathds{1}_B$ with $P_{\textbf{Rat}} \mathds{1}_B$. We can write $$P_{\textbf{Rat}}\mathds{1}_B = \mu(B) + \sum_{\chi \in R^{\ast}_q} c_{\chi} \rho_{\chi} + \sum_{\chi \in \textbf{Rat} \setminus R_q} c_{\chi} \rho_{\chi}  $$ where $\rho_{\chi}$ is a $\chi$-eigenfunction of norm $1$ and $c_{\chi} \in \bC$. From Proposition~\ref{prop: equidist implies small proj} we get that \begin{align}\label{Rq estimate}\left\| \frac{1}{|S_N|} \sum_{a \in S_N} T^a \sum_{\chi \in R^{\ast}_q} c_{\chi} \rho_{\chi} \right\|^2_2 \leq \left\| \sum_{\chi \in R^{\ast}_q} c_{\chi} \rho_{\chi} \right\|^2_2 \leq (2\delta+\delta^2) \mu(B)^2 \leq 3\delta.\end{align} Now by Theorem~\ref{thm: exponential sums} we get that $$ \limsup_{N \to \infty} \left| \frac{1}{|S_N|} \sum_{a \in S_N} \chi(a) \right| \leq \eta \quad{ }\text{ for all } \chi \in \textbf{Rat} \setminus R_q.$$ This implies that \begin{align*} \limsup_{N \to \infty} \left\| \frac{1}{|S_{N}|} \sum_{a \in S_{N}} T^a \sum_{\chi \in \textbf{Rat} \setminus R_q} c_{\chi} \rho_{\chi}\right \|^2_2 &=  \limsup_{N \to \infty} \left\| \sum_{\chi \in \textbf{Rat} \setminus R_q}\left(\frac{1}{|S_{N}|} \sum_{a \in S_{N}} \chi(a) \right) c_{\chi} \rho_{\chi}\right \|^2_2 \\ &\leq \eta^2  \sum_{\chi \in \textbf{Rat} \setminus R_q} c^2_{\chi} \\&\leq \eta^2 \mu(B)^2 \leq \eta^2. \end{align*} Finally, combining this estimate with (\ref{Rq estimate}) and using the triangle inequality gives the desired estimate (\ref{spherical mean estimate}). \end{proof}

\section{Spherical Recurrence} 
\label{sec: spherical recurrence}

We are now in a position to prove Proposition~\ref{prop: ergodic magyar}. By Lemma~\ref{lemma: measure increment}, this reduces to the following.
\begin{prop}\label{prop: spherical recurrence} Fix $\epsilon>0$ and let $T: \bZ^d \curvearrowright (X,\mu)$ be an ergodic measure preserving action. Suppose that $B \subset X$ is $(q_{\eta},\delta)$-equidistributed where $\eta<\frac{1}{2}\epsilon$ and $\delta=\frac{1}{3}\eta^2$. Then there exists $N_0$ such that for all $N \geq N_0$ we have that $$\frac{1}{|S_N|} \sum_{a \in S_N} \mu(B \cap T^aB) > \mu(B)^2 - \epsilon.$$ 

\end{prop}

\begin{proof} From Cauchy-Schwartz followed by the Spherical Mean Ergodic theorem (Theorem~\ref{spherical mean ergodic thm}) we get that \begin{align*} |\frac{1}{|S_N|} \sum_{a \in S_N} \mu(B \cap T^aB) - \mu(B)^2| &=|\langle \mathds{1}_B,  \frac{1}{|S_N|} \sum_{a \in S_N}T^a \mathds{1}_B-\mu(B) \rangle | \\ &\leq \left \| \mathds{1}_B \right\|_2 \left\| \frac{1}{|S_N|} \sum_{a \in S_N}T^a\mathds{1}_B - \mu(B) \right \|_2 \\ &< \epsilon    \end{align*} for sufficiently large integers $N$. \end{proof}

\section{Locally isometric embeddings of trees} 

We now turn to proving our main recurrence result (Theorem~\ref{thm: intro tree recurrence}). To this end, it will be useful to relax the notion of a locally isometric embedding, introduced in Definition~\ref{def: intro trees}, to the broader notion of a \textit{locally isometric immersion}.

\begin{mydef}[Locally isometric immersions] Recall that an edge-labelled tree is a a tuple $\tau=(V,E,\phi,L)$ where $(V,L)$ is a finite tree (connected acyclic graph) with vertex set $V$ and edge set $E$ and $\phi:E \to L$ is a function to a set $L$. If $L \subset \bZ_{>0}$ then a \textit{locally isometric immersion} of $\tau$ into $\bZ^d$ is a map $\iota: V \to \bZ^d$ such that for each edge $e=\{v_1,v_2\} \in E$ we have that $$\|\iota(v_1)-\iota(v_2)\|^2=\phi(e).$$ As per Definition~\ref{def: intro trees}, we say that $\iota$ is a \textit{locally isometric embedding} if it is injective.  

\end{mydef}

By Lemma~\ref{lemma: measure increment}, the following result implies Theorem~\ref{thm: intro tree recurrence}.

\begin{thm} \label{thm: tree recurrence} Let $\epsilon>0$ and suppose that $T:\bZ^d \curvearrowright (X,\mu)$ is an Ergodic measure preserving action and $B \subset X$ is $(q_{\eta},\delta)$ equidistributed with $\eta<\frac{1}{2}\epsilon$ and $\delta=\frac{1}{3}\eta^2$. Then there exists $N_1$ such that for all edge-labelled trees $\tau=(V,E,\phi,\bZ_{>N_1})$, with $|V|=m$, there exists a locally isometric embedding $\iota$ of $\tau$ into $\bZ^d$ such that $$\mu\left( \bigcap_{v \in V} T^{\iota(v)}B \right) \geq \mu(B)^m - m\epsilon. $$

\end{thm}

Before we embark on the proof, let us introduce the notion of a rooted edge-labelled tree.

\begin{mydef} A \textit{rooted edge-labelled tree} is a tuple $\tau=(V,v_0,E,\phi,L)$ where $(V,E,\phi,L)$ is an edge-labelled tree and $v_0 \in V$ is a distinguished vertex, which we call the root of $\tau$.

\end{mydef}

It will be convenient to use the averaging notation $$\mathbb{E}_{a \in A} f(a)=\frac{1}{|A|} \sum_{a \in A} f(a)$$ for finite sets $A$. 

\begin{proof}[Proof of Theorem~\ref{thm: tree recurrence}] Choose, by the Spherical Mean Ergodic theorem (Theorem~\ref{spherical mean ergodic thm}), a positive integer $N_0>0$ such that \begin{align}\label{spherical mean 1} \left\|  \bE_{a \in S_N} T^a \mathds{1}_B - \mu(B) \right\|_2 < \epsilon \textrm{ } \text{ for all } N> N_0. \end{align}

 Fix a rooted edge-labelled tree $\tau=(V,v_0,E,\phi,\bZ_{>N_0})$. We let $$\cI=\cI(\tau) = \{ \iota:V \to \bZ^d  \text{ } | \text{ } \iota \text{ is a locally isometric immersion of $\tau$ with } \iota(v_0)=0. \}.$$ We now aim to show that \begin{align}\label{tree correlation estimate} \left \| \mathbb{E}_{ \iota \in \cI} \prod_{v\in V \setminus \{v_0\}} T^{\iota(v)} \mathds{1}_B - \mu(B)^{m-1} \right \|_{L^2(X,\mu)} \leq (m-1) \epsilon. \end{align}

This may be proven by induction on $|V|=m$ as follows: The $m=2$ case is precisely the estimate (\ref{spherical mean 1}). Now suppose $m>2$ and let $e^*=\{v_1,v^*\}$ be an edge of $\tau$ such that $v^*$ is a leaf (i.e., $e^*$ is the only edge which contains it) and $v^*\neq v_0$. Now consider the rooted edge-labelled tree $\tau'=\tau - v^*$ obtained by deleting this leaf, more precisely $$\tau'=(V \setminus \{v^*\},v_0, E \setminus \{e^*\}, \phi|_{E \setminus \{e^*\}}, \bZ_{>N_0}),$$  and let $\mathcal{I}'=\mathcal{I}(\tau')$ be the corresponding set of immersions. We have the recursion 

\begin{align*} \mathbb{E}_{ \iota \in \cI} \prod_{v\in V \setminus \{v_0\}} T^{\iota(v)} \mathds{1}_B &= \mathbb{E}_{ \iota \in \cI'} \left( \left( \prod_{v\in V \setminus \{v_0,v^*\}} T^{\iota(v)} \mathds{1}_B \right) \left( \bE_{a \in S_{\phi(e^*)}} T^{\iota(v_1)+ a} \mathds{1}_B  \right) \right). \\ &=  \mu(B) \cdot \mathbb{E}_{ \iota \in \cI'}\left( \prod_{v\in V \setminus \{v_0,v^*\}} T^{\iota(v)} \mathds{1}_B \right) \\&+ \mathbb{E}_{ \iota \in \cI'} \left( \left( \prod_{v\in V \setminus \{v_0,v^*\}} T^{\iota(v)} \mathds{1}_B \right) \left( \bE_{a \in S_{\phi(e^*)}} T^{\iota(v_1)+ a} \mathds{1}_B - \mu(B)  \right) \right). \end{align*}

But the $L^2$-norm of second term is at most \begin{align*} \left\| \bE_{a \in S_{\phi(e^*)}} T^{\iota(v_1)+ a} \mathds{1}_B  - \mu(B) \right \|_{L^2(X,\mu)}  = \left\|  \bE_{a \in S_{\phi(e^*)}} T^{a} \mathds{1}_B  - \mu(B) \right \|_{L^2(X,\mu)} < \epsilon \end{align*} where in the equality we only used the fact that $T^{\iota(v_1)}$ is an isometry that fixes constant functions. Combining this estimate with the recursion and the inductive hypothesis finally completes the induction step, and thus establishes (\ref{tree correlation estimate}). Now let $\cE=\cE(\tau) \subset \cI(\tau)$ be those elements of $\cI$ which are embeddings.  Note that $$|\cI|=\prod_{e \in E} |S_{\phi(e)}|$$ and that $$ |\cE| \geq \prod_{j=1}^{|E|} (|S_{\phi(e_j)}|-j+1),$$ where $e_1,e_2, \ldots, e_{m-1}$ is some enumeration of $E$ such that $e_1$ contains $v_0$ and the subgraph with edges $e_1, \ldots, e_i$ is connected for all $i=1, \ldots, m-1$. This means that, as $N_0 \to \infty$, an arbitrarily large proportion of elements of $\cI$ are embeddings. More precisely, we have the uniform bound $$\frac{|\cE|}{|\cI|} \geq \left( 1 - \frac{m}{\sqrt{N_0}} \right)^m \to 1 \text{ as }N_0 \to \infty.$$ This means that we may choose $N_1=N_1(B,m,\epsilon)>N_0$ (it does not depend on the tree $\tau$, only on its size) such that for all trees of the form $\tau=(V,v_0,E,\phi,\bZ_{>N_1})$, with $|V|=m$,  we have that $$ \| \bE_{\iota \in \cE}  \prod_{v\in V \setminus\{v_0\}} T^{\iota(v)} \mathds{1}_B - \bE_{\iota \in \cI} \prod_{v\in V \setminus\{v_0\}} T^{\iota(v)} \mathds{1}_B\| < \epsilon,$$  where $\cE=\cE(\tau)$ and $\cI=\cI(\tau)$. Combining this with (\ref{tree correlation estimate}) we obtain that $$\left \| \mathbb{E}_{ \iota \in \cE} \prod_{v\in V \setminus\{v_0\}} T^{\iota(v)} \mathds{1}_B - \mu(B)^{m-1} \right \|_{L^2(X,\mu)} < m \epsilon. $$ Now Cauchy Schwartz gives (by the same argument as in the proof of Proposition~\ref{prop: spherical recurrence})  $$\left |\mathbb{E}_{ \iota \in \cE}\text{ } \mu \left ( \prod_{v\in V} T^{\iota(v)} B \right) - \mu(B)^{m} \right | < m \epsilon.$$ It now immediately follows that we may choose an embedding $\iota \in \cE$ satisfying the conclusion of the theorem. \end{proof}

\section{Optimal Pointwise Recurrence}
\label{sec: optimal recurrence}

We now establish a pointwise recurrence result, whose combinatorial consequence recovers a recent \textit{optimal unpinned distance} result obtained by Lyall and Magyar (Theorem 1 and Theorem 3 in \cite{LyallMagyar}). In what follows, we shall use the shorthand ``$P(b)$ holds for $\mu$-many $b \in B$'' to mean $\mu(\{ b \in B \text{ } | \text{ } P(b) \})>0$. 

\begin{prop}[Pointwise recurrence] Fix $\epsilon>0$. Let $T: \bZ^d \curvearrowright (X,\mu)$ be an ergodic measure preserving action and suppose that $B \subset X$, with $\mu(B)> \epsilon$, is $(q_{\eta},\delta)$-equidistributed  where $\eta<\frac{1}{2}\epsilon^2 $ and $\delta=\frac{1}{3}\eta^2$. Then there exists $N_0$ such that for all $N \geq N_0$ we have that $$\frac{1}{|S_N|} \sum_{a \in S_N} \mathds{1}_{B}(T^{qa}b) > \mu(B) - \epsilon \quad{ } \text{ for $\mu$-many} \text{ }  b \in B.$$ 

\end{prop}

More generally, we have the following pointwise multiple recurrence result.

\begin{prop}[Pointwise multiple recurrence] Fix $\epsilon>0$ and a positive integer $m$.  Let $T: \bZ^d \curvearrowright (X,\mu)$ be an ergodic measure preserving action and suppose that $B \subset X$, with $\mu(B)>\epsilon$, is $(q_{\eta},\delta)$-equidistributed where $\eta<\frac{1}{2}\epsilon^2 m^{-1}$ and $\delta=\frac{1}{3}\eta^2$. Then there exists a positive integer $N_0$ such that for all $N_0<N_1< \cdots < N_m$ we have that there exists $\mu$-many $b \in B$ such that for all $j \in \{1, \ldots, m \}$ we have that $$\frac{1}{|S_{N_j}|} \sum_{a \in S_{N_j}} \mathds{1}_{T^aB}(b) > \mu(B) - \epsilon.$$

\end{prop}

\begin{remark} The combinatorial consequence of this result is is similair to a recent \textit{optimal pinned distances} result obtained by Lyall and Magyar (Theorem 2 and Theorem 4 in \cite{LyallMagyar}). The difference is that they instead have the hypothesis $\eta <<\epsilon^3$, where the implied constant does not depend on $m$, and in fact no other parameter depends on $m$ (in particular, $N_0$). It would be interesting to see whether this can be proven in the Ergodic setting.

\end{remark}

\begin{proof} We let $$U_N= \{ x \in X \text{ }: \text{ } |\mu(B) - \frac{1}{|S_N|} \sum_{a \in S_N} T^a \mathds{1}_B | \geq \epsilon \}.$$ We have that \begin{align*} \mu(U_N) & \leq \frac{1}{\epsilon} \left \| \mu(B) - \frac{1}{|S_N|} \sum_{a \in S_N} T^a \mathds{1}_B \right\|_1 \\ & \leq \frac{1}{\epsilon}\left \| \mu(B) - \frac{1}{|S_N|} \sum_{a \in S_N} T^a \mathds{1}_B \right\|_2\end{align*}  and so by Theorem~\ref{spherical mean ergodic thm} we have that$$\limsup_{N \to \infty} \mu(U_N) \leq \frac{1}{\epsilon} \cdot 2\eta  < \frac{\epsilon}{m}. $$ Hence for all sufficiently large $N_1, \ldots, N_m$ we have that $$\mu\left(B \setminus \bigcup_{j=1}^m U_{N_j}\right) >\epsilon - m \cdot \frac{\epsilon}{m}=0.$$ \end{proof}

\appendix
\section{Sets invariant under finite index subgroups}

For the sake of completeness, we provide a simple proof of the claims made in Definition~\ref{def: ergodic components of finite index} regarding the $T^Q$-invariant sets.

\begin{lemma} \label{lemma: sigma algebra without small sets is finite} Let $(X,\mathcal{A},\mu)$ be a probability space such that $\mathcal{A}$ does \textbf{not} have sets of arbitrarily small positive measure, i.e., $$ r:=\inf\{ \mu(A) \text{ } | A \in \mathcal{A}, \text{ } \mu(A)>0 \} >0. $$ Then the $\sigma$-algebra $\mathcal{A}$ is finite (modulo null sets, as always).

\end{lemma}

\begin{proof} Let $r$ be as in the statement of the Lemma. Take $A_0 \in \mathcal{A}$ with $\mu(A_0)<\frac{3r}{2}$. If there exists an $\mathcal{A}$-measurable $A' \subset A_0$ with $0<\mu(A')< \mu(A_0)$ then $\mu(A') \geq r$ and so $\mu(A_0 \setminus A') \leq \frac{r}{2}$, contradicting the definition of $r$. Hence it follows that $A_0$ is an atom of $\mathcal{A}$. By considering the space $X \setminus A_0$, we may continue in this way to obtain a partition of $X$ into a finite union of atoms of $\mathcal{A}$. \end{proof}

\begin{prop} Let $G \curvearrowright (X,\mu)$ be an ergodic measure preserving action of a group $G$. Then for finite index normal subgroups $H \lhd  G$ there exists a $H$-invariant $C \subset X$ such that $H \curvearrowright (C, \mu( \cdot | C))$ is ergodic and a finite set $g_1, \ldots, g_n \in G$ such that the collection of sets $\{ g_i C \text{ } | \text{ } i =1, \ldots, n\}$ is a partition of $X$.

\end{prop}

\begin{proof} Let $\mathcal{A} = \{ A \subset X \text{ } | \text{ } hA=A \text{ for all } h \in H \}$ be the $\sigma$-algebra of $H$ invariant sets. Notice that since $H$ is normal in $G$, we have that $\mathcal{A}$ is $G$-invariant (i.e., $gA \in \mathcal{A}$ for all $g \in G$ and $A \in \mathcal{A}$). Futhermore, as $H$ fixes each $A \in \mathcal{A}$, there is a natural action of $G/H$ on $\mathcal{A}$. From the ergodicity of $G$, it follows that for non-null $A \in \mathcal{A}$ we have that $$X= \bigcup_{g \in G} gA= \bigcup_{u \in G/H} uA$$ and so $\mu(A) \geq \frac{1}{|G/H|} >0.$ Applying Lemma~\ref{lemma: sigma algebra without small sets is finite} above we may take an atom $C \in \mathcal{A}$ of positive measure. Since $C$ does not contain any non-trivial element of $\mathcal{A}$, it follows that the action of $H$ on $C$ is ergodic, as desired. Any translate of $C$ is also an atom of $\mathcal{A}$, hence the distinct translates of $C$ are disjoint. Moreover, they cover $X$ by the ergodicity of the action of $G$. \end{proof}

\end{document}